\documentclass[<preprint>]{elsarticle}
\usepackage{lineno,hyperref}
\modulolinenumbers[50]
\usepackage[figuresright]{rotating}

\usepackage{amsfonts}
\usepackage{amsmath,amssymb,amsthm}
\usepackage{enumerate}
\usepackage{faktor}
\usepackage{tabularx}
\usepackage{marvosym}

\usepackage[applemac]{inputenc}
\newtheorem{thm}{Theorem}[section]

\newtheorem{lem}{Lemma}[section]
\newtheorem{pro}{Proposition}[section]

\newtheorem{cor}{Corollary}[section]

\newtheorem{rem}{Remark}[section]

\newcommand{\B}{\mathcal{B}}

\newcommand{\C}{\mathbb{C}}

\begin{document}
\begin{frontmatter}
\title{ Product  commuting maps  with the  $\lambda$-Aluthge transform }
\author{Fadil Chabbabi}
\ead{Fadil.Chabbabi@math.univ-lille1.fr}
\address{ Universit\'e  Lille1, UFR de Math\'ematiques,
\\ Laboratoire CNRS-UMR 8524 P. Painlev\'e, 
59655 Villeneuve d'Ascq  Cedex, France}


\begin{abstract}
Let $H$ and $K$ be two Hilbert spaces and $\B(H)$ be the algebra of all bounded linear operators from $H$ into itself. The main purpose of this paper is to obtain a characterization of bijective  maps $\Phi  :\B(H) \to \B(K)$   satisfying the following condition   
\begin{equation*}
\Delta_{\lambda}(\Phi (A)\Phi (B))=\Phi (\Delta_{\lambda}(A B)) \quad  for all \; \; A, B\in \B(H),
 \end{equation*}
 where   $\Delta_{\lambda}(T)$ stands the $\lambda$-Aluthge transform of the operator $T\in\B(H)$. 
 
  More precisely, we prove that a bijective map $\Phi $ satisfies  the above condition, if and only if $\Phi (A)=UAU^*$ for all $A\in\B(H)$, for some unitary operator $U:H\to K$.
\end{abstract}

\begin{keyword}
\texttt{Normal, Quasi-normal operators, Polar decomposition, $\lambda$-Aluthge transform.}
\end{keyword}

\end{frontmatter}


 \section{Introduction }
Let $H$ and $K$ be  two complex Hilbert spaces  and  $\mathcal{B}(H, K)$ be the Banach space of all bounded linear operators from $H$ into $K$. In the case $K=H$, $\mathcal{B}(H,H)$  is simply denoted by  $\B(H) $ which is a  Banach algebra.   For $T\in \mathcal{B}(H,K)$,  we set $\mathcal{R}(T)$ and $\mathcal{N}(T)$  for the range and the null-space of $T$, respectively. We also denote by  $T^* \in \mathcal{B}(K,H)$  the adjoint operator of $T$.

The spectrum of an operator $T\in\B(H)$ is denoted by $\sigma(T)$ and  $W(T)$  is the numerical range  of $T$.
 
 An operator $T\in \mathcal{B}(H,K)$ is  a {\it partial isometry} when $T^*T$ is an orthogonal projection (or, equivalently $TT^*T = T$). 
  In particular $T$ is an {\it isometry} if $T^*T=I$, and {\it unitary} if $T$ is a surjective isometry. 
  
The polar decomposition of $T\in\B(H)$ is given by  $T=V|T|$, where $|T|=\sqrt{T^*T}$ and $V$ is an appropriate partial isometry such that  $\mathcal{N}(T)=\mathcal{N}(V)$ and $\mathcal{N}(T^*)=\mathcal{N}(V^*)$. 
 
The Aluthge transform  introduced in \cite{MR1047771}  as $\Delta(T)=|T|^{\frac{1}{2}}V|T|^{\frac{1}{2}}$  to    extend some properties  of hyponormal  operators. 
Later, in \cite{MR1997382},    Okubo introduced a more general notion called  $\lambda-$Aluthge
transform which has also been studied in detail.

 For $\lambda \in [ 0, 1]$, the $\lambda-$Aluthge transform is  defined by, 
 $$
 \Delta_{\lambda}(T)=|T|^{\lambda}V|T|^{1-\lambda}.
 $$
  Notice  that $\Delta_0(T) =V|T|=T$, and $\Delta_1(T) = |T|V$ which is known as Duggal's
  transform.  It has since been studied in many different contexts and considered by a
  number of authors (see for instance,  \cite{ MR2013473, MR2148169, MR1780122, MR1829514, MR1971744,  MR1873788} and some
  of the references there).  The interest of the Aluthge transform lies in the fact that
   it respects many   properties of the original operator. For example, 
\begin{equation}
   \sigma_*(\Delta_{\lambda}(T)) = \sigma_*(T),   \mbox{ for every } \; \; T \in \B(H), 
\end{equation} 
  
where $  \sigma_*$ runs over a large family of spectra. See \cite[Theorems 1.3, 1.5]{MR1780122}.

 Another important property is that $Lat(T)$, the lattice of $T$-invariant subspaces of
 $H$, is nontrivial if and only if $Lat(\Delta(T))$ is nontrivial
(see \cite[Theorem 1.15]{MR1780122}). 

  Recently in \cite{MR3455750}, F.Bothelho, L.Moln\'ar and G.Nagy   studied  the   linear  bijective mapping on Von Neumann algebras which commutes with the $\lambda$-Aluthge transforms. They focus of bijective linear maps  such that $$\Delta_{\lambda}(\Phi (T))=\Phi (\Delta_{\lambda}(T )) \mbox{ for   every }  T \in \B(H). $$

  We are concerned in this paper with the more general problem of   product  commuting maps  with the  $\lambda$-Aluthge transform  in the following sense,
  \begin{equation}\label{c1}
\Delta_{\lambda}(\Phi (A)\Phi (B))=\Phi (\Delta_{\lambda}(A B))   \mbox{ for   every }   A, B\in \B(H),
 \end{equation}
for some fixed $\lambda \in]0,1[$. \\

Our main result gives a complete description of the bijective map $\Phi :\B(H)\to\B(K)$ which satisfies Condition (\ref{c1}) and is stated as follows.  
 \begin{thm}\label{th1} 
 Let $H$ and $K$ be a  complex Hilbert spaces, with $H$  of dimension greater than $2$.  Let    $\Phi  :\B(H)\to\B(K)$ be  bijective.  
 Then,

 $\Phi $ satisfies  $\left(\ref{c1}\right)$, if and only if, there exists  an unitary operator $U:H\to K$ such that  
 $$\Phi (A)=UAU^*\quad \text{ for all} \; A\in\B(H).$$
 \end{thm}\

\begin{rem}
$(1)$ In one dimensional, the result of Theorem \ref{th1} fails, as given in the following example: let the map $\Phi: \C\to\C$ defined by  $${\Phi(z)=\left\{\begin{array}{l}\frac{1}{z}\text{ if } z\ne 0,\\ 0 \text{ if } z=0. \end{array}\right.}$$
Clearly $\Phi$ is bijective  and satisfies $\left(\ref{c1}\right)$, but it is not additive.

$(2)$ The  map  $\Phi $   considered in our theorem is not assumed to satisfy any kind of continuity. However, an automatic continuity is obtained as a consequence.\\
\end{rem}
 The proof of Theorem \ref{th1} is stated in next section. Several auxiliary results are needed for the proof and are established below.

\section{ Proof of the main theorem} 

We first recall some basic notions that are used in the sequel. An operator $T\in\B(H)$  is  normal if $T^*T=TT^*$,  and  is  quasi-normal, if  it commutes with $T^*T$ ( i.e. $TT^*T=T^*T^2$), or equivalently  $|T|$ and $V$ commutes. In finite dimensional spaces  every quasi-normal operator  is normal.  It is easy to see that if $T$ is quasi-normal, then $T^2$ is also quasi-normal, but the converse is false as shown by nonzero  nilpotent operators .

Also, quasi-normal operators  are exactly the fixed points of  $\Delta_{\lambda}$ (see \cite[Proposition 1.10]{MR1780122}).
   \begin{equation}\label{qn}
 T \; \; \mbox{ quasi-normal} \; \;  \iff  \Delta_{\lambda}(T) = T.
 \end{equation}
 An  idempotent self adjoint  operator $P\in\B(H)$  is said to be an orthogonal projection.  Clearly  quasi-normal idempotents  are orthogonal projections. \\
Two  projections $P,Q\in\B(H)$ are said to be orthogonal if $PQ=QP=0$ and we denote $P\perp Q$. 
  A partial ordering between orthogonal projections is defined as follows,  $$P\leq Q\:  \mbox{  if  } \:  PQ = QP = P.$$
We start with  the following lemma, which  gives the "only if" part in our theorem. It has already been mentioned in other papers in the case $H=K$ (see \cite{MR2148169}, for example). We give the proof for completeness.

\begin{lem}\label{l0} Let $U:H\to K$ be an unitary operator, and $\lambda\in [0,1]$. We have the following identity 
 $${ \Delta_{\lambda}}(UTU^*)=U{ \Delta_{\lambda}}(T)U^*, \quad \mbox{for every} \; T\in\B(H).$$
\end{lem}
\begin{proof}
Let $T\in\B(H)$.  It is easy to check
$$ \vert UTU^*\vert = U\vert T\vert U^* \quad \mbox{and} \quad  \vert UTU^*\vert^{\lambda} = U\vert T\vert^{\lambda}U^*, \; \; \lambda\in [0,1].$$
Now, let $T=V\vert T\vert$ be a polar decomposition. Then
$$UTU^*=UV\vert T\vert U^*=(UVU^*)(U\vert T\vert U^*)=\tilde{V}\vert UTU^*\vert,$$
where $\tilde{V}=UVU^*$. $\tilde{V}$ is a partial isometry,  $\mathcal{N}(UTU^*) = \mathcal{N}(\tilde{V})$  and hence $\tilde{V}\vert UTU^*\vert$ is the polar decomposition of $UTU^*$. This implies that :
\begin{align*}
{ \Delta_{\lambda}}(UTU^*) &=  |UTU^*|^{\lambda}\tilde{V}|UTU^*|^{1-\lambda}  \\
 &= U|T|^{\lambda}U^* \tilde{V}U|T|^{1-\lambda}U^*\\
   &= U|T|^{\lambda}V|T|^{1-\lambda}U^*\\
   & = U{ \Delta_{\lambda}}(T)U^*.
\end{align*}
This completes the proof.
\end{proof}  

 For $x,y \in H$ , we denote by  $x\otimes y$ the  at most rank one operator defined by $$(x\otimes y)u=<u,y>x \: \mbox{   for } \:  u\in H.$$ It is easy to show that every rank one   operator has  the previous form and that   $x\otimes y$ is an orthogonal projection, if and only if $x=y$ and $\|x\|=1$. We have the following proposition,
 \begin{pro}\label{p1} Let $x,y \in H$ be nonzero vectors. We have 
 $$\Delta_{\lambda}(x\otimes y)=\dfrac{<x,y>}{\|y\|^2}(y\otimes y)\: \:  \mbox{ for every }  \: \lambda \in ]0,1[
.$$
\end{pro}
\begin{proof}
Denote $T=x\otimes y$, then  $T^*T=|T|^2=\|x\|^2(y\otimes y)=\big(\frac{\|x\|}{\|y\|}(y\otimes y)\big)^2\;\;\text{and}\;\;|T|=\sqrt{T^*T}=\frac{\|x\|}{\|y\|}(y\otimes y).$
It follows that   $|T|^2=\|x\|\|y\| |T|\;\; \text{and}\;\; |T|^\gamma =(\|x\|\|y\|)^{\gamma-1}|T|$ for every  $\gamma>0$.

Now,  let  $T=U|T|$ be the polar decomposition of $T$, we have 
\begin{align*}
\Delta_{\lambda}(T)&=|T|^{\lambda}U|T|^{1-\lambda}\\&=(\|x\|\|y\|)^{{\lambda-1}}(\|x\|\|y\|)^{{-\lambda}}|T|U|T|\\&=\dfrac{1}{\|x\|\|y\|}|T|T\\&=\frac{1}{\|y\|^2}(y\otimes y)\circ (x\otimes y)=\dfrac{<x,y>}{\|y\|^2}(y\otimes y).
\end{align*}
\end{proof}

We deduce the next 

\begin{cor}\label{rt}Let $R$ be a bounded linear operator on $H$ and $\lambda \in ]0,1[$. Suppose that $$\Delta_{\lambda}(RT)= \Delta_{\lambda}(TR), $$ for every rank one operator of  the form $T=y\otimes y$. Then, there exists some $\alpha \in \C$ such that $R =\alpha I$.
\end{cor}
\begin{proof} Denote $A=R^*$. First, we claim that the linear operator $A$ satisfies the property that for every $z\in H$ we either have $Az$ is orthogonal to $z$ (calling $z$ being of the first kind) or $Az,z$ are linearly dependent (calling $z$ being of the second kind). Indeed, let $z\in H$ and $T=z\otimes z$, from the assumption and the Proposition \ref{p1},   we have
$$<Rz,z>z\otimes z=\Delta_\lambda(Rz\otimes z)=\Delta_\lambda(z\otimes Az).$$
In the case when $<Rz,z>=0$, then $z$ is of the first kind. And if $<Rz,z>\ne 0$ then $Az\ne 0$, and from  the last equality it follows that 
$$<Rz,z>z\otimes z=\dfrac{<Rz,z>}{\|Az\|^2}Az\otimes Az.$$
Thus $Az$ and $z$ are linearly dependent.  

Now, $A$ is a scalar multiple of the identity. Indeed, on contrary assume that we have vector $x$ which is of the first kind but not of the second kind and that we have a vector $y$ which is of the second kind but not of the first kind. Then $x,y$ are linearly independent. We may assume that $Ay=y$. Set $x'=Ax$. For a real number $t$  from the unit interval and for $z_t=tx+(1-t)y$ we have $Az_t=tx'+(1-t)y$. It is clear that  the equation $<Az_t,z_t>=t(1-t)\big( <x',y>+<y,x>\big)+(1-t)^2\|y\|^2=0$ has at most one solutions $t_1\in ]0,1[$. Also, with the fact that $x,y$ are linearly independent, then  $Az_t,z_t$ are linearly independent for all $t\in]0,1[$ except for at most one $t\in ]0,1[$. So, for example,  for small enough positive $t$ the vector $z_t$ does not of the first kind nor of the second. 

 This shows that either have that $Az$ is orthogonal to $z$ for all vectors $z$ or we have $Az,z$ are linearly dependent for all vectors $z$. In the first case we have that $A=0$, in the second one $A$ is a scalar multiple of the identity. In any way $A$ is a scalar multiple of the identity. Thus $R=A^*=\alpha I$ for some $\alpha \in\C$.

\end{proof}

The following lemma,  provides  a  criterion for an operator  to be positive  through its $\lambda$-Aluthge transform. It will play a  crucial  role in the  proof of  Theorem \ref{th1}.

 \begin{lem}\label{l1}
  Let $T\in \B(H)$ be an invertible operator. The following conditions are equivalent :
  \begin{enumerate}[(i)]
  \item $T$ is positive;
  \item for every  $\lambda \in [0,1],  {\Delta_{\lambda}}(T)$  is  positive;
  \item there exists   $\lambda \in [0,1]$ such that ${\Delta_{\lambda}}(T)$ is positive. 
 \end{enumerate}
 In particular,  ${\Delta_{\lambda}}(T)=cI$ for some nonzero scalar $c$, if and only if
 $T=cI$. 
 \end{lem}
 \begin{proof} 
The implications $(i)\Rightarrow (ii)\Rightarrow (iii)$ are  trivial. It remains to show that  $(iii)\Rightarrow (i)$. Let us consider the polar decomposition  $T=U|T|$ of $T$ and assume that ${\Delta_{\lambda}}(T)$ is a positive operator. Since $T$ invertible it follows that $|T|^{1-\lambda}$ is invertible and $U$ is unitary. We claim  that $U = I$. Indeed, let us denote $A=|T|^{2\lambda-1}$, we have    
  \begin{align*}
  AU &=|T|^{2\lambda-1}U\\&=|T|^{\lambda-1}(|T|^{\lambda}U|T|^{1-\lambda})|T|^{\lambda-1}\\
  &=|T|^{\lambda-1}{\Delta_{\lambda}}(T)|T|^{\lambda-1}.
  \end{align*}
  This follows that  $AU=|T|^{\lambda-1}{\Delta_{\lambda}}(T)|T|^{\lambda-1}$ is positive. In particular it is self adjoint. Thus  $AU=(AU)^*=U^*A$ and then $UAU=A$.
  Therefore  $(AU)^2=A^2$. It follows that $AU=A$ since  $AU$ and $A$ are positive. Thus   $U=I$ and this gives $T=U|T|=|T|$ is positive. 
 \end{proof}
 
\begin{rem}  The assumption  $T$ is invertible is necessary in  the previous lemma.
 Indeed, let $ T=x\otimes y$, with $x, y$ be  nonzero independent  vectors such that $<x,y>\ge 0$. Using  proposition \ref{p1}, we get ${\Delta_{\lambda}}(T)$ is positive while $T$ is not.
\end{rem}

\begin{lem}\label{l2}
Let $T\in\B(H)$ be an arbitrary operator and $P\in\B(H)$ be an orthogonal projection. The following are equivalent :
\begin{enumerate}[(i)]
\item  $\Delta_{\lambda}(TP)=T$ ;
\item $TP=PT=T\:   \mbox{ and  }   T\: \mbox{ is quasi-normal}.$
\end{enumerate}
\end{lem}
 \begin{proof} The implication $(ii)\Rightarrow (i)$  is obvious. We show the direct implication. 
Consider  $TP=U|TP|$ the polar decomposition of $TP$. Suppose that $\Delta_{\lambda}(TP)=T$,  then 
  \begin{equation}
  |TP|^{\lambda}U|TP|^{1-\lambda}=T \mbox{~~~ and ~~~} |TP|^{1-\lambda}U^*|TP|^{\lambda}=T^*.
  \end{equation} 
  It follows that $$\mathcal{R}(T)\subseteq \mathcal{R}(|TP|^{\lambda})\subseteq \overline{\mathcal{R}(|TP|^2)}$$
  and 
   $$\mathcal{R}(T^*)\subseteq \mathcal{R}(|TP|^{1-\lambda})\subseteq \overline{\mathcal{R}(|TP|^2)}.$$
  In the other hand, we have $|TP|^2=PT^*TP=P|T|^2P$. Thus  $\overline{\mathcal{R}(|TP|^2)}\subseteq \mathcal{R}(P)$. Hence $\mathcal{R}(T)\subset \mathcal{R}(P)$ and $\mathcal{R}(T^*)\subset \mathcal{R}(P)$. Which implies that $PT=T$ and $PT^*=T^*$. Therefore  $$PT=TP=T\;\;\; \text{and}\;\; T \: \mbox{  is quasi-normal}.$$
 \end{proof}

\begin{pro}\label{l5} Let $\Phi $ be a bijective  map satisfying $\left(\ref{c1}\right)$.  Then\\
$$\Phi (0)=0.$$
Moreover,  there exists a bijective function $h:\C\to\C$ such that:
\begin{enumerate}[(i)]
\item $\Phi (\alpha I)=h(\alpha)I$ for all $\alpha \in\C$.
\item $h(\alpha \beta)=h(\alpha)h(\beta)$ for all $\alpha,\beta \in \C$.
\item $h(1)=1$ and  $h(-\alpha)=-h(\alpha)$ for all $\alpha \in \C$.
\end{enumerate}
\end{pro}
\begin{proof} For the first assertion, since $\Phi$ is bijective, there exists  $A\in \B(H)$ such that $\Phi (A)=0$. Therefore $\Phi (0)=\Delta_{\lambda}(\Phi (A)\Phi (0))=0$.

Let us show now that there exists a function $h : \C\to\C$ such that  $\Phi (\alpha I)=h(\alpha)I$ for all $\alpha \in\C$.
If $\alpha=0$ the function $h$ is defined by $h(0)=0$ since $\Phi (0)=0$. Now, suppose  that $\alpha$ is a nonzero scalar and  denote by $R=\Phi (\alpha I)$ in particular $R\ne 0$. From $\left(\ref{c1}\right)$ it follows that
\begin{equation}
\Delta_{\lambda}(R\Phi (A))=\Phi (\Delta_{\lambda}(\alpha A))=\Delta_{\lambda}(\Phi (A\alpha I))=\Delta_{\lambda}(\Phi (A)R),
\end{equation}
for every $A\in\B(H)$.
Since $\Phi $ is onto, then  $\Delta_{\lambda}(RT)=\Delta_{\lambda}(TR)$ for  every rank one  operator of the form $T=y\otimes y$ from $\B(K)$.  Since $R=\Phi (\alpha I)$ different from zero and by Corollary \ref{rt}, there exists a nonzero scalar $h(\alpha) \in \C$ such that $R=\Phi (\alpha I)=h(\alpha) I$. In the other hand,  $\Phi $ is bijective and its inverse $\Phi ^{-1}$ satisfies the same condition as $\Phi $. It follows that the map $h: \C\to   \C$ is well defined and it is  bijective.  

Moreover, using again Condition $\left(\ref{c1}\right)$,  we get 
$$h(\alpha\beta)I=\Delta_{\lambda}(\Phi (\alpha\beta I))=\Delta_{\lambda}(\Phi (\alpha I)\Phi (\beta I))=h(\alpha)h(\beta)I,$$
for every $\alpha, \beta \in \C$ and therefore $h$ is multiplicative. 

Since  $(h(1))^2=h(1)$  and   $h$ is bijective  with $h(0)=0$, we obtain $h(1)=1$ . Similarly $h(-1)=-1$, thus  $h(-\alpha)=h(-1)h(\alpha)=-h(\alpha)$ for all $\alpha\in\C$.
\end{proof}
As a direct consequence   we have the following corollary :
\begin{cor}\label{cor1}
 Let $\Phi :\B(H)\to\B(K)$ be a bijective map satisfying $\left(\ref{c1}\right)$. Then 
 \begin{enumerate}[(i)]
   \item  $\Phi (I)=I$.
  \item $\Delta_{\lambda}\circ \Phi =\Phi \circ \Delta_{\lambda}$. In particular, $\Phi $ preserves the set of quasi-normal operators in both directions.
 \item$\Phi (\alpha A)=h(\alpha)\Phi (A)$ for all $\alpha$ and $A$ quasi-normal.
  \end{enumerate}
 \end{cor}
 The following lemma gives some properties of  bijective  maps satisfying $\left(\ref{c1}\right)$.
  \begin{lem}\label{l6} Let $\Phi $ be a bijective  map satisfying $\left(\ref{c1}\right)$. Then 
\begin{enumerate}[(1)]
\item $\Phi (A^2)=(\Phi (A))^2$ for all  $A$ quasi-normal.
\item $\Phi $ preserves the set of  orthogonal projections.
\item $\Phi $ preserves the orthogonality between the projections ; $$P\perp Q \Leftrightarrow \Phi (P)\perp\Phi (Q).$$
\item $\Phi $ preserves the order relation on the set of orthogonal projections in the both directions ; $$Q\leq P \Leftrightarrow \Phi (Q) \leq \Phi  (P).$$
\item  $\Phi (P+Q)=\Phi (P)+\Phi (Q)$ for all orthogonal projections $P,Q$ such that $P\perp Q$. 
\item  $\Phi $ preserves the set of  rank one  orthogonal projections in the both directions.
\end{enumerate}
 \end{lem}  
 \begin{proof}
 
 $(1)$  From  $\left(\ref{c1}\right)$ , we have  $\Delta_{\lambda}((\Phi (A))^2)=\Phi (\Delta_{\lambda} (A^2))$ for every operator $A$. 
Let  $A$ be a quasi-normal operator;  since $\Phi $ preserves the set of quasi-normal operators, we get 
  $\Phi (A), \Phi (A^2), (\Phi (A))^2 $ are  quasi-normal.  It follows from \eqref{qn}) that $  \Delta_{\lambda} (A^2)= A^2$ and $  \Delta_{\lambda} (\Phi (A^2))=\Phi (A^2)$. We deduce that 
 $$(\Phi (A))^2 = \Delta_{\lambda}((\Phi (A))^2)=\Phi (\Delta_{\lambda} (A^2))=\Phi (A^2).$$

 $(2)$ Follows from the first assertion since orthogonal projections are  quasi-normal.
 
 $(3)$ Assume that $P,Q$ are orthogonal and denote  $N=\Phi (P)$ and $M=\Phi (Q)$.  From  $\left(\ref{c1}\right)$ we have,  
 $\Delta_{\lambda}(MN)=\Delta_{\lambda}(NM)=0$ and using  \cite[Theorem 4]{MR2392831}, we obtain 
 $$(MN)^{2}=MNMN=0 \mbox{~~ and~~~} (NM)^{2}=NMNM=0.$$
It follows that, 
$$\|MN\|^2 = \|(MN)^*MN\|=  \|NMN\|=  \|(NMN)^2\|^{\frac{1}{2}}= \|NMNMN\|^{\frac{1}{2}}=  0$$
and similarly, $NM = 0$.

Finally  $\Phi $ preserves the orthogonality between the projections.
 
 $(4)$ Now, assume that $Q \leq P$, then $PQ=QP=Q$. By   $\left(\ref{c1}\right)$ we have $$\Delta_{\lambda}(\Phi (Q)\Phi (P))=\Phi (Q).$$ By  Lemma \ref{l2}, we get   $\Phi (Q)\Phi (P)=\Phi (P)\Phi (Q)=\Phi (Q)$ since $\Phi (P)$ is an orthogonal projection.  Therefore  $\Phi (Q) \leq \Phi  (P)$. Since $\Phi $ is bijective and its inverse satisfies the same conditions as $\Phi $, hence $\Phi $ preserves the order relation between the  projections in both directions. 
 
 $(5)$ Suppose that $P,Q$ are orthogonal. We have  $P \leq P+Q$ and $Q \leq P+Q$. Which gives   $\Phi (P)  \leq \Phi (P+Q)$ and $\Phi (Q)  \leq \Phi (P+Q)$. From $\Phi (P)  \perp \Phi (Q)$, it follows that $$\Phi (P)+\Phi (Q)\leq \Phi (P+Q).$$ Since $\Phi ^{-1}$ satisfies the same assumptions as $\Phi $, we have 
$$ \begin{array}{lll}
 \Phi (P+Q) &=& \Phi [\Phi^{-1}(\Phi(P) ) +\Phi^{-1}(\Phi(Q) )]\\
&\le &   \Phi [\Phi^{-1}(\Phi(P)  +\Phi(Q) )]\\
 &= &  \Phi (P)+\Phi (Q).
\end{array}
 $$
Finally $\Phi (P+Q)=\Phi (P)+\Phi (Q)$.\\

 $(6)$
 Let $P=x\otimes x$ be a rank one  projection. We claim that  $\Phi (P)$ is a non zero minimal projection. Indeed, let $y\in K$ be an unit vector such that $y\otimes y \leq \Phi (P)$.  Thus $\Phi ^{-1}(y\otimes y) \leq P$.  Since $P$ is a minimal projection and $\Phi ^{-1}(y\otimes y)$ is a non zero projection, then $\Phi ^{-1}(y\otimes y)= P$. Therefore  $\Phi (P)=y\otimes y$ is a rank one projection. 
 \end{proof}

We now  prove the following lemma which is needed in  the proof of our result.
  \begin{lem}\label{l7}
  Let $\Phi $ be a bijective  map satisfying  $\left(\ref{c1}\right)$. Let $P=x\otimes x, Q=x'\otimes x'$ be  rank one  projections such that $P\perp Q$. Then  
  $$\Phi (\alpha P+\beta Q)=h(\alpha)\Phi (P)+h(\beta)\Phi (Q)$$ 
 for  every  $\alpha, \beta\in\C$.
  \end{lem}
  \begin{proof}
   If $\alpha=0$ or $\beta=0$ the result is trivial. Suppose that $\alpha\not =0$ and
  $\beta\not=0$.
Clearly $\alpha P+\beta Q$ is normal, hence  $\Phi (\alpha P+\beta Q)$ is quasi-normal. 
By  Condition $\left(\ref{c1}\right)$ we get 
 \begin{eqnarray*} 
 \Phi (\alpha P+\beta Q)&=&\Delta_{\lambda}(\Phi (\alpha P+\beta Q))\\&=&\Phi
  (\Delta_{\lambda}(\alpha P+\beta Q))\\&=&\Phi \Big(\Delta_{\lambda}\big((\alpha P+\beta
  Q)
 (P+Q)\big)\Big)\\&=&\Delta_{\lambda}\big(\Phi (\alpha P+\beta Q)\Phi (P+Q)\big)\\
&=&\Phi (\alpha P+\beta Q)\Phi (P+Q).
 \end{eqnarray*}
  Since  $\Phi (P+Q)=\Phi (P)+\Phi (Q)$ is a an orthogonal projection, 
  hence by Lemma \ref{l2} 
   \begin{eqnarray*}
  \Phi (\alpha P+\beta Q)&=&\Phi (\alpha P+\beta Q)(\Phi (P)+\Phi (Q))=(\Phi (P)+\Phi
 (Q))\Phi (\alpha P+\beta Q)\\
  &=&(\Phi (P)+\Phi (Q))\Phi (\alpha P+\beta Q)(\Phi (P)+\Phi (Q)).
   \end{eqnarray*}
  Denote by $T=\Phi (\alpha P+\beta Q)$. We write   $\Phi (x\otimes x)=y\otimes y$ and
   $\Phi (x'\otimes x')=y'\otimes y'$  with $ y\perp y'$, since  $\Phi $ preserves  orthogonality and rank one projections. 
   We have,
    $$T=(y\otimes y+y'\otimes y')T(y\otimes y+y'\otimes y').$$
Hence
    \begin{equation}\label{eq1}
   T= <Ty,y>y\otimes y+ <Ty',y>y\otimes y'+ <Ty,y'>y'\otimes y+ <Ty',y'>y'\otimes y'. 
   \end{equation}
  We  show that $<Ty',y>=<Ty,y'>=0$ by using   $\left(\ref{c1}\right)$
    \begin{eqnarray*}
    \Delta_{\lambda}(\Phi (\alpha P+\beta Q)\Phi (P))&=& \Phi (\Delta_{\lambda}((\alpha P+\beta Q)P))\\&=& \Phi (\alpha P)=h(\alpha)\Phi (P).
      \end{eqnarray*}
 In other terms, we write 
 $$\Delta_{\lambda}(Ty\otimes y)=\Delta_{\lambda}(y\otimes T^*y )=h(\alpha) y\otimes y.$$
Since $h(\alpha)\not=0$, then  $T^*y\not=0$. By Proposition \ref{p1} follows that
 $$<Ty,y>y\otimes y =\dfrac{<y,T^*y>}{\|T^*y\|^2}T^*y\otimes T^*y=h(\alpha) y\otimes y.$$
Therefore $<Ty,y>=h(\alpha)$ and $T^*y=\overline{h(\alpha)}y$.  Using   (\ref{eq1})  we
 deduce   $$T^*y=<T^*y,y>y+<T^*y,y'>y'=\overline{h(\alpha)}y.$$ It follows that
  $<Ty',y>=<T^*y,y'>=0$. 
 
 By similar arguments we   get   $<Ty',y'>=h(\beta)$ and $<Ty,y'>=0$. Again  (\ref{eq1})
 implies that 
 $$ \Phi (\alpha P+\beta Q)=T=h(\alpha)y\otimes y+h(\beta)y'\otimes y'=h(\alpha)\Phi
 (P)+h(\beta)\Phi (Q).$$
   \end{proof}

    Now, we are in a position to   prove our main result 
 \bigskip

 {\it   \bf Proof of Theorem \ref{th1}.} The "only if" part is an immediate consequence of Lemma \ref{l0}. \\  

We show the "if" part. Assume that $\Phi :\B(H)\to\B(K)$ is bijective and  satisfies  $\left(\ref{c1}\right)$.   The proof of theorem is organized in several steps.
 
 \begin{enumerate}[Step 1.]
 \item  For every $A\in\B(H)$,  we have 
 \begin{equation}\label{eq}
<\Phi (A)y,y>=h(<Ax,x>) \;\;\text{for all unit vectors $x,y$ such $\Phi (x\otimes x)=y\otimes y$}.
  \end{equation}
   Let  $x,y\in H$  be   unit vectors such that  $\Phi (x\otimes x)=y\otimes y$.  From  $\left(\ref{c1}\right)$, we obtain
 \begin{eqnarray*}   \Delta_{\lambda}(\Phi (A)y\otimes y)&=& \Delta_{\lambda}(\Phi (A)\Phi (x\otimes x))\\
 &=& \Phi (\Delta_{\lambda}(A(x\otimes x)))\\
 &=& \Phi (\Delta_{\lambda}(Ax\otimes x)).
 \end{eqnarray*} 
Using  Proposition \ref{p1}, we get  
  $$<\Phi (A)y,y>y\otimes y=\Phi (<Ax,x>x\otimes x)=h(<Ax,x>)y\otimes y.$$
 It follows that $$<\Phi (A)y,y>=h(<Ax,x>).$$
  
  \item The function $h$ is additive.
  
  Let $P=x\otimes x, Q=x'\otimes x'$  are  rank one projections such that $P\bot Q$ and 
   $\alpha, \beta \in\C$. Denote by $z=\frac{1}{\sqrt{2}}(x+x')$,  then $\|z\|=1$ and  $\|Pz\|^2= \|Qz\|^2=\frac{1}{{2}}$.  Note $z\otimes z$ is rank one  projection, then there exist an unit vector $u\in K$ such  that $\Phi (z\otimes z)=u \otimes u$. We take $A=\alpha P+\beta Q$ in the identity (\ref{eq}), we get that 

 \begin{eqnarray*} <\Phi (\alpha P+ \beta Q)u,u>&=&h(<\alpha Pz +\beta Q z,z>)\\ 
&=&h(\alpha \|Pz\|^2+\beta \|Qz \|^2)\\
 &=&h(\frac{1}{{2}})h(\alpha+\beta). 
 \end{eqnarray*}
Thus 
  \begin{equation}\label{eqq}
  <\Phi (\alpha P+ \beta Q)u,u>= h(\frac{1}{{2}})h(\alpha+\beta).
     \end{equation}
 In the other hand, by  Lemma \ref{l7} we have 
  $$\Phi (\alpha P+ \beta Q)=\Phi (\alpha P)+\Phi (\beta Q)=h(\alpha)\Phi (P)+ h(\beta) \Phi (Q).$$
And therefore  
  \begin{eqnarray*}
  <\Phi (\alpha P+\beta Q)u,u>&=&<\Phi (\alpha P)u+\Phi (\beta Q)u,u>\\&=&<\Phi (\alpha P)u,u>+<\Phi (\beta Q)u,u>\\&=& h(<\alpha Pz,z>)+h(<\beta Qz,z>)\\&=&h(\alpha \|Pz\|^2)+h(\beta \|Qz\|^2)\\&=& h(\frac{1}{2})(h(\alpha)+h(\beta)).
  \end{eqnarray*}
 Using  (\ref{eqq}) and the preceding equality, it follows that $$h(\frac{1}{{2}})h(\alpha+\beta)=h(\frac{1}{{2}})(h(\alpha)+h(\beta)).$$ 
Now $h(\frac{1}{{2}})\neq 0$ gives $$h(\alpha+\beta)=h(\alpha)+h(\beta).$$
  
  \item $h$ is continuous. 
  
Let   $\mathcal{E}$ be a bounded subset  in $\C$ and  $A\in\B(H)$  such that  $\mathcal{E}\subset W(A)$.

By   (\ref{eq}),  
 \begin{equation*}
h(\mathcal{E})\subset h(W(A))   = W(\Phi (A))
 \end{equation*}
  Now , $W(\Phi (A))$ is bounded and thus  $h$ is bounded on the bounded subset. With the fact that $h$ is additive and multiplicative, it then follows that $h$ is continuous  (see, for example, \cite{MR804038}). We derive that $h$ is a continuous  automorphism over the complex field $\C$. It follows that $h$ is the identity or the  complex conjugation map. \\
  
\item The map $\Phi $ is linear or anti-linear. 
  
 Let  $y\in K$ and $x\in H$ be two unit vectors, such that $y \otimes y=\Phi ( x\otimes x)$. Let $\alpha\in \C$ and $A ,B\in \B(H)$ be arbitrary. Using (\ref{eq}), we get   
  \begin{eqnarray*}
 <\Phi (A+B)y,y>&=&h(<(A+B)x,x>)\\&=&h(<Ax,x>+<Bx,x>)\\&=&h(<Ax,x>)+h(<Bx,x>)
 \\&=&<\Phi (A)y,y>+<\Phi (B)y,y>\\&=&<(\Phi (A)+\Phi (B))y,y>,
  \end{eqnarray*}
and 
\begin{eqnarray*}
 <\Phi (\alpha A)y,y>&=&h(<\alpha A x,x>)=h(\alpha)h(<Ax,x>)=h(\alpha)<\Phi (A)y,y>.
\end{eqnarray*}
  Therefore  $$<\Phi (A+B)y,y>=<(\Phi (A)+\Phi (B))y,y>\;\;\text{and} <\Phi (\alpha A)y,y>=h(\alpha)<\Phi (A)y,y>,$$
   for all unit vectors $y\in K$. It follows that  $\Phi (A+B)=\Phi (A)+\Phi (B)$ and $\Phi (\alpha A)=h(\alpha)\Phi (A)$ for all $A,B\in\B(H)$. Therefore  $\Phi $ is linear or anti-linear since $h$ is the identity or the complex conjugation. \\
  
\item There exists an unitary  operator  $U\in \B(H,K)$,  such that $\Phi (A)=UAU^*$ for every $A\in \B(H)$. 

Let $A\in\B(H)$ be  invertible. By   $\left(\ref{c1}\right)$, we have 
 $$\Delta_{\lambda}(\Phi (A)\Phi (A^{-1}))=\Delta_{\lambda}(\Phi (A^{-1})\Phi (A))=\Phi (\Delta_{\lambda}(I))=I.$$
 By Lemma \ref{l1}, we get that  $$\Phi (A)\Phi (A^{-1})=\Phi (A^{-1})\Phi (A)=I.$$
 It follows that $\Phi (A)$ is also invertible and $(\Phi (A))^{-1}=\Phi (A^{-1})$. Therefore $\Phi $  preserves the set of  invertible operators. By \cite[Corollary 4.3]{MR2669430}, there exists a bounded linear and bijective  operator 
 $V:H\to K$ such that  $\Phi $ takes one of the following form 
\begin{equation}\label{f1}
\Phi (A)=VAV^{-1}\;\; \;  \text{for all}\;\; A\in\B(H)
\end{equation}
or 
\begin{equation}\label{f2}
\Phi (A)=VA^*V^{-1}  \;\; \;  \text{for all}\;\;  A\in\B(H)
\end{equation}
 In order to complete the proof we have to show that $V$ is  unitary  and $\Phi $ has form (\ref{f1}).
 
  First, we show that  $V:H\to K$ in (\ref{f1}) ( or in (\ref{f2}))  is necessarily  unitary. Indeed, let $x\in H$ be a unit vector. We know that $x\otimes x$ is an orthogonal projection, hence  $\Phi (x\otimes x)=Vx\otimes (V^{-1})^*x$ is also an orthogonal projection.  It follows that $ (V^{-1})^*x=Vx$ for all unit vector $x\in H$ and then  $ (V^{-1})^*=V$.  Therefore  $V$ is unitary.

 Seeking contradiction,  we suppose that (\ref{f2}) holds. Multiplying  (\ref{f2}) by  $V^*$   left and by  $V$  right, since $\Phi$ commutes with ${\Delta_{\lambda}}$, we obtain
\begin{equation}\label{ee}
{\Delta_{\lambda }}(A^{*})= ({\Delta_{\lambda}}(A))^{*},\quad \mbox{ for every} \; A\in\B(H).
\end{equation}
Let us consider  $A=x\otimes x'$ with $x, x'$  are  unit independent vectors  in $H$. $A^*=x'\otimes x$ and  by Proposition \ref{p1}, we have 
 $$ {\Delta_{\lambda}}(A)=<x,x'>(x'\otimes x') \; \;  \mbox{and} \; \; {\Delta_{\lambda}}(A^{*})=<x',x>(x\otimes x),$$
which contradicts  (\ref{ee}).   This completes the proof.\\
 \end{enumerate}

{\bf  Acknowledgments.}

 I wish to thank Professor Mostafa Mbekhta  for the  interesting discussions as well as his useful suggestions for the improvement of this paper. Also, I thank the referee for valuable comments that helped to improve the paper, in particular the proof of Corollary 2.1.

This work was supported in part by the Labex CEMPI (ANR-11-LABX-0007-01).\\

{\bf References}

\bibliographystyle{abbrv}

\bibliography{fadil.bib}

\begin{thebibliography}{10}

\bibitem{MR1047771}
A.~Aluthge.
\newblock On {$p$}-hyponormal operators for {$0<p<1$}.
\newblock {\em Integral Equations Operator Theory}, 13(3):307--315, 1990.

\bibitem{MR2013473}
T.~Ando and T.~Yamazaki.
\newblock The iterated {A}luthge transforms of a 2-by-2 matrix converge.
\newblock {\em Linear Algebra Appl.}, 375:299--309, 2003.

\bibitem{MR2148169}
J.~Antezana, P.~Massey, and D.~Stojanoff.
\newblock {$\lambda$}-{A}luthge transforms and {S}chatten ideals.
\newblock {\em Linear Algebra Appl.}, 405:177--199, 2005.

\bibitem{MR3455750}
F.~Botelho, L.~Moln{\'a}r, and G.~Nagy.
\newblock Linear bijections on von {N}eumann factors commuting with
  {$\lambda$}-{A}luthge transform.
\newblock {\em Bull. Lond. Math. Soc.}, 48(1):74--84, 2016.

\bibitem{MR2669430}
N.~Boudi and M.~Mbekhta.
\newblock Additive maps preserving strongly generalized inverses.
\newblock {\em J. Operator Theory}, 64(1):117--130, 2010.

\bibitem{MR2392831}
S.~R. Garcia.
\newblock Aluthge transforms of complex symmetric operators.
\newblock {\em Integral Equations Operator Theory}, 60(3):357--367, 2008.

\bibitem{MR1780122}
I.~B. Jung, E.~Ko, and C.~Pearcy.
\newblock Aluthge transforms of operators.
\newblock {\em Integral Equations Operator Theory}, 37(4):437--448, 2000.

\bibitem{MR1829514}
I.~B. Jung, E.~Ko, and C.~Pearcy.
\newblock Spectral pictures of {A}luthge transforms of operators.
\newblock {\em Integral Equations Operator Theory}, 40(1):52--60, 2001.

\bibitem{MR1971744}
I.~B. Jung, E.~Ko, and C.~Pearcy.
\newblock The iterated {A}luthge transform of an operator.
\newblock {\em Integral Equations Operator Theory}, 45(4):375--387, 2003.

\bibitem{MR804038}
R.~R. Kallman and F.~W. Simmons.
\newblock A theorem on planar continua and an application to automorphisms of
  the field of complex numbers.
\newblock {\em Topology Appl.}, 20(3):251--255, 1985.

\bibitem{MR1997382}
K.~Okubo.
\newblock On weakly unitarily invariant norm and the {A}luthge transformation.
\newblock {\em Linear Algebra Appl.}, 371:369--375, 2003.

\bibitem{MR1873788}
T.~Yamazaki.
\newblock An expression of spectral radius via {A}luthge transformation.
\newblock {\em Proc. Amer. Math. Soc.}, 130(4):1131--1137 (electronic), 2002.

\end{thebibliography}

\end{document}